\documentclass[11pt, leqno]{amsart}
\usepackage[T1]{fontenc}
\usepackage[utf8]{inputenc}
\usepackage{amsfonts,amssymb, amscd}
\usepackage{amsmath}
\usepackage{lmodern}
\usepackage{mathrsfs} 
\usepackage{mathtools}
\usepackage{amsthm}
\usepackage{qsymbols}
\usepackage{graphicx}
\usepackage{latexsym}
\usepackage{chngcntr}
\usepackage{tikz-cd}
\usepackage[noadjust]{cite}
\usepackage{paralist}
\usepackage[parfill]{parskip}

\usepackage{enumitem}
\newtheorem{theorem}{Theorem}
\newtheorem{lemma}[theorem]{Lemma}
\newtheorem{proposition}[theorem]{Proposition}
\newtheorem{corollary}[theorem]{Corollary}
\theoremstyle{definition}
\newtheorem{definition}[theorem]{Definition}

\newenvironment{tagprop}[1]{%
  \renewcommand\thetheorem{#1}
  \proposition
}{\endtheorem}
\usepackage[plainpages=false,pdfpagelabels,citecolor=red]{hyperref}
\usepackage{xcolor}
\hypersetup{
 colorlinks,
 linkcolor={cyan!90!black},
 citecolor={magenta},
 urlcolor={green!40!black}
} 
\newcommand{\R}{\mathbb{R}}
\newcommand{\IC}{\mathbb{C}}
\newcommand{\IN}{\mathbb{N}}
\newcommand{\IZ}{\mathbb{Z}}

\newcommand{\cJ}{\mathcal{J}}

\newcommand{\V}{V}
\renewcommand{\L}{\mathrm{L}}
\newcommand{\C}{\mathrm{C}}
\newcommand{\W}{\mathrm{W}}
\renewcommand{\S}{\mathrm{S}}
\newcommand{\dualp}[1]{|#1|^{p/2-1}}
\newcommand{\ind}{{\mathbf{1}}}
\newcommand{\Pb}{P_{B^p}}
\newcommand{\e}{\mathrm{e}}

\renewcommand{\i}{\mathrm{i}}
\renewcommand{\d}{\mathrm{d}}
\newcommand{\eps}{\varepsilon}
\newcommand{\loc}{\mathrm{loc}}
\renewcommand\Re{\operatorname{Re}}
\renewcommand\Im{\operatorname{Im}}
\renewcommand{\tilde}{\widetilde}
\newcommand{\biggmid}{\mathrel{\bigg|}}

\newcommand{\Lop}{\mathcal{L}}
\newcommand{\cl}[1]{\overline{#1}}
\renewcommand\div{\operatorname{div}}
\DeclareMathOperator{\bd}{\partial}

\DeclareMathOperator*{\essinf}{essinf}
\DeclareMathOperator{\dist}{d}

\DeclareMathOperator{\sgn}{sgn}
\DeclareMathOperator{\dom}{\mathcal{D}}
\author{Moritz Egert}
\address{Laboratoire de Math\'{e}matiques d'Orsay, Univ. Paris-Sud, CNRS, Universit\'{e} Paris-Saclay, 91405 Orsay, France}
\email{moritz.egert@math.u-psud.fr}
\thanks{}
\date{\today}
\dedicatory{}
\title{On $p\,$-elliptic divergence form operators and holomorphic semigroups}
\begin{document}
\begin{abstract}
Second order divergence form operators are studied on an open set with various boundary conditions. It is shown that the $p$-ellipticity condition of Carbonaro--Dragi\v{c}ević and Dindo\v{s}--Pipher implies extrapolation to a holomorphic semigroup on Lebesgue spaces in a $p$-dependent range of exponents that extends the maximal range for general strictly elliptic coefficients. This has immediate consequences for the harmonic analysis of such operators, including $\mathrm{H}^\infty$-calculi and Riesz transforms.
\end{abstract}
\maketitle
\section{Introduction and main results}
\label{Sec: Introduction}

Let $A: O \to \Lop(\IC^d)$ be a measurable strictly elliptic matrix function on an open set $O \subseteq \R^d$, that is to say, there are constants $\lambda, \Lambda > 0$ such that for almost every $x \in O$,
\begin{align}
\label{eq: ellipticity}
 |A(x) \xi| \leq \Lambda |\xi| \qquad \text{and} \qquad \Re (A(x)\xi \mid \xi) \geq \lambda |\xi|^2
\end{align}
hold for all $\xi \in \IC^d$. The associated divergence form operator $L = -\div(A \nabla \, \cdot)$ is defined on $\L^2(O)$ in the weak sense through a sesquilinear form
\begin{align}
\label{eq: sesquilinear form}
a: \V \times \V \to \IC, \quad a(u,v) = \int_O A \nabla u \cdot \cl{\nabla v} \; \d x,
\end{align}
where (mixed) Dirichlet and Neumann boundary conditions are incorporated through the choice of the form domain $\W^{1,2}_0(O) \subseteq \V \subseteq \W^{1,2}(O)$, see Section~\ref{Sec: Form domain} below. It is well known that $L$ is maximal accretive~\cite[VI.6]{Kato} and therefore $-L$ generates a holomorphic $C_0$-semigroup of contractions $T = (T(t))_{t \geq 0}$ on $\L^2(O)$ of angle $\pi/2 - \omega$, where
\begin{align}
\label{eq: angle omega}
\omega \coloneqq \sup_{u \in \V} \arg a(u,u),
\end{align}
see~\cite[XI.6, Thm.~1.24]{Kato}. In this paper, we are concerned with extrapolating $T$ by density to a holomorphic $C_0$-semigroup on $\L^q(O)$ for $q$ in an interval around $2$.

For real coefficient matrices, it has been long known that $T$ extrapolates to $\L^q(O)$ for all $q \in [1,\infty)$, see, for example, \cite{Arendt-terElst, Ouhabaz, Rehberg-terElst}. For complex matrices, on the contrary, there is a natural threshold in $q$ that is related to Sobolev embeddings. For $|1/2 - 1/q| < 1/d$ extrapolation is well-known by various different proofs on various classes of open sets~\cite{Pascal-Memoir, JDE, Ouhabaz, Patrick, Davies-JFA}. In the plane this covers the full range $q \in (1,\infty)$, but in dimensions $d \geq 3$ the semigroup may cease from extrapolating if $|1/2 - 1/q| > 1/d$, even on $O = \R^d$, see \cite[Prop.~2.10]{HMM}.

In their groundbreaking paper \cite{Mazya-Cialdea}, Maz'ya and Cialdea have considered operators with coefficients $A \in \C^1(\cl{O} \to \Lop(\IC^d))$ and pure Dirichlet boundary conditions on a bounded regular domain. They have found an algebraic condition on the matrix $A$ that is sufficient for $T$ to extrapolate to a \emph{contraction} semigroup on $\L^p(O)$ and that is also necessary for the latter to hold if, in addition, $\Im(A)$ is symmetric. This was generalized to the setup described above, but still for pure Dirichlet boundary conditions, by Carbonaro and Dragi\v{c}ević~\cite{Carbonaro-Dragicevic}. They have elegantly rephrased the condition in \cite{Mazya-Cialdea} as
\begin{align}
\label{eq: Deltap}
\Delta_p(A) \coloneqq \essinf \limits_{x \in O} \min \limits_{\xi \in \IC^d, \; |\xi|=1} \Re(A(x)\xi, \cJ_p \xi) \geq 0,
\end{align}
where $\cJ_p: \IC^d \to \IC^d$ is the $\R$-linear map defined by
\begin{align}
\label{eq: Jp}
\cJ_p(\alpha + \i \beta) = 2 \bigg(\frac{\alpha}{p'} + \frac{\i \beta}{p}\bigg) \qquad (\alpha, \beta \in \R^d)
\end{align}
and $p' = p/(p-1)$ is the H\"older conjugate of $p \in (1,\infty)$. The stronger condition $\Delta_p(A) > 0$ was also introduced in \cite{Carbonaro-Dragicevic} in the context of dimension-free bilinear embeddings and independently by Dindo\v{s}--Pipher \cite{Dindos-Pipher} in the context of boundary value problems. We shall follow terminology of \cite{Carbonaro-Dragicevic} and say that $A$ is \emph{$p$-elliptic} in this case. By definition, this means for almost every $x \in O$ the $p$-adapted lower bound
\begin{align}
\label{eq: p-elliptic definition}
 \Re(A(x)\xi, \cJ_p \xi) \geq \Delta_p(A) |\xi|^2 \quad \text{with} \quad \Delta_p(A)>0
\end{align}
for all $\xi \in \IC^d$. This can be viewed as an interpolating condition between general strictly elliptic and strictly elliptic real matrices. Indeed, \eqref{eq: ellipticity} automatically implies $\Delta_2(A) \geq \lambda$ and if in addition $A$ is real, then we have $\Delta_p(A) \geq \lambda(2/p' \wedge 2/p)$ for every $p \in (1,\infty)$. 

Throughout the paper, as far as $p$-ellipticity is concerned, we assume $p \in (1,\infty)$. This will imply that all appearing Lebesgue exponents $q$ also belong to this open interval.

\subsection{Main results}
\label{Subsec: Main results}

Our main result gives an extension of the extrapolation range for operators with $p$-elliptic coefficients that bridges the gap between the optimal ranges for general complex and real coefficients. In view of the discussion above this is only of interest in dimension $d \geq 3$. What remains open is the question whether this range is optimal over the class of all $p$-elliptic matrices.

We shall work under the assumption of (Sobolev) embedding properties of the form domain $\V$ that are made precise in Section~\ref{Subsec: embedding properties}. Here, we only mention that the homogeneous version holds without any restrictions on $O$ in case of pure Dirichlet conditions and for mixed boundary conditions if $O$ is bounded, connected, and Lipschitz regular around the Neumann boundary part. 

\begin{theorem}
\label{thm: main1}
Let $d \geq 3$ and assume that $\V$ has the embedding property. If $A$ is $p$-elliptic, then for every $\eps > 0$ the semigroup $(\e^{-\eps t}T(t))_{t \geq 0}$ generated by $-L-\eps$ extrapolates to a $C_0$-semigroup on $\L^q(O)$ provided that
\begin{align*}
 \big|1/2-1/q\big| \leq 1/d + \big(1-2/d \big) \big| 1/2-1/p\big|.
\end{align*}
This semigroup is bounded holomorphic of angle $\pi/2 - \omega$. If $\V$ has the homogeneous embedding property, then the same result also holds for $\eps =0$.
\end{theorem}

We adopt the notion of bounded holomorphic semigroups from \cite{Engel-Nagel}. Uniform boundedness is required on every sector of angle $\psi < \pi/2-\omega$ but with  possibly $\psi$-dependent bound. The extrapolation to a $C_0$-semigroup in Theorem~\ref{thm: main1} was known previously only under the structural assumption that $\Im A$ is symmetric~\cite[Thm.~1.4]{ELSV}. See also \cite{ELSV, Ouhabaz, Ouhabaz-Gaussian, Mourou-Selmi} for earlier contributions. We remark that Theorem~\ref{thm: main1} applies in particular to $q = p$ and $q=p'$.

The proof will be given at the end of Section~\ref{Sec: Extension of the interval}. It follows a two-step procedure: If we can extrapolate the semigroup to $\L^p(O)$, then we can use ultracontractivity to extrapolate further to the range of $q$'s in Theorem~\ref{thm: main1}. An essential tool in this approach are $\L^2$ off-diagonal bounds for $T$ that we reproduce in Section~\ref{Subsec: OD} for convenience. The required extrapolation to $\L^p(O)$ in turn relies on some of the fundamental algebraic calculations in \cite{Mazya-Cialdea}. In fact, we obtain as our second main result in Section~\ref{Sec: Dissipativity}, and without any further assumptions on $V$ besides its mere definition at the start of Section~\ref{Sec: Form domain}, the following

\begin{theorem}
\label{thm: main2}
If $\Delta_p(A) \geq 0$, then $T$ extrapolates to a holomorphic $C_0$-semigroup of contractions on $\L^q(O)$ provided that $|1/2-1/q| < |1/2-1/p|$. If $A$ is $p$-elliptic, then these properties also hold at the endpoints $q=p$ and $q=p'$.
\end{theorem}

By this means, we generalize the implication ``$(a) \Rightarrow (b)$'' in \cite[Thm.~1.3]{Carbonaro-Dragicevic} to more general boundary conditions.

\subsection{Consequences}
\label{Subsec: Consequences}

Let us recall the important observation of \cite[Prop.~5.15]{Carbonaro-Dragicevic} that $p$-ellipticity of $A$ can equivalently be stated through the inequality
\begin{align*}
 \mu(A) \coloneqq \essinf \limits_{x \in O} \inf \limits_{\xi \in \IC^d, (A(x)\xi \mid \cl{\xi}) \neq 0}  \frac{\Re(A(x)\xi \mid \xi)}{|(A(x)\xi \mid \cl{\xi})|} > |1-2/p|,
\end{align*}
which decouples $p$ and $A$. For convenience, we have included a direct proof in appendix. We conclude that the set $\{p : \text{$A$ is $p$-elliptic}\}$ is open in $(1,\infty)$. This being said, our main results reveal some new features even for the classical strictly elliptic matrices as in \eqref{eq: ellipticity}. Indeed, we obtain from the trivial bound $\mu(A) \geq \lambda/\Lambda$ that every strictly elliptic matrix is $p$-elliptic provided that $|1/2-1/p| < \lambda/(2\Lambda)$ and we conclude from Theorem~\ref{thm: main1} the following

\begin{corollary}
\label{cor: main1}
Let $d \geq 3$ and assume that $\V$ has the homogeneous embedding property. Then $T$ extrapolates to a bounded holomorphic $C_0$-semigroup on $\L^q(O)$ of angle $\pi/2 - \omega$ provided that
\begin{align*}
 \big|1/2-1/q\big| < 1/d + \big(1/2-1/d \big) \lambda/\Lambda.
\end{align*}
\end{corollary}

Extrapolation beyond the range $|1/2 - 1/q| \leq 1/d$ is essentially known in this context. Our proof here, however, appears particularly clean in that it avoids the sophisticated self-improvement properties of  invertibility in complex interpolation scales~\cite{Pascal-Memoir, JDE} or reverse Hölder inequalities~\cite{Patrick} and, as a consequence, pinpoints the improvement in terms of ellipticity.

An independent interest in Theorem~\ref{thm: main1} stems from its immediate consequences for the harmonic analysis of $L$ on $\L^q(O)$, such as $\mathrm{H}^\infty$-calculus, Riesz transforms and Kato square root estimates, at least when $O$ is either the whole space or a bounded connected set that is Lipschitz regular around the Neumann boundary part. Indeed, such results come for free once the extrapolation of the semigroup has been settled and we refer the reader to \cite{Pascal-Memoir, JDE} for a precise account.

\subsection*{Acknowledgments}

We are grateful to Andrea Carbonaro for inspiring discussions during a workshop at CIRM in April 2018 and to Oliver Dragi\v{c}ević for sharing further insight on and around the topic. We acknowledge that a result similar to our Theorem~\ref{thm: main1} has been discovered independently in \cite{EHRT} by ter Elst, Haller-Dintelmann, Rehberg, and Tolksdorf through a different method. This research has been supported by the ANR project RAGE ANR-18-CE40-0012-01.
\section{Notation}
\label{Sec: Notation}

We use $|\cdot|$ for both the Euclidean norm on $\IC^d$, $d \geq 1$, and the operator norm of matrices viewed as linear operators on the Hilbert space $\IC^d$. Inner products $(\cdot \mid \cdot)$ on complex Hilbert spaces are linear in the first component. We use the same symbol for the $\L^p - \L^{p'}$ duality pairing that extends the inner product on $\L^2$. We denote the (semi)-distance between subsets $E, F$ of $\R^d$ by $\dist(E,F)$. Given $z \in \IC$, we write 
\begin{align*}
 \sgn z = \begin{cases}
           z/|z| &\quad \text{if $z \neq 0$,} \\ 0 &\quad \text{if $z=0$}.
          \end{cases}
\end{align*}
For $\psi \in (0,\pi)$, we define the open sector $\S_\psi^+ \coloneqq \{z \in \IC \setminus \{0\}: |\arg(z)| < \psi\}$ of opening angle $2\psi$ symmetric about $(0,\infty)$ and for convenience we put $\S_0^+ \coloneqq (0,\infty)$. Occasionally, we use the symbol $\lesssim$ for inequalities that hold up to multiplication by a constant. None of our constants depends on the structure of $A$ itself but only on quantified parameters such as $\lambda, \Lambda, \Delta_p(A)$.

All function spaces in this paper are over the complex numbers. By $\W^{1,2}(O)$, we denote the usual Sobolev space of functions $u \in \L^2(O)$ with distributional gradient $\nabla u \in \L^2(O)^d$ equipped with the Hilbertian norm $\|u\|_{1,2} \coloneqq (\|u\|_2^2 + \|\nabla u\|_2^2)^{1/2}$. Given a closed set $D \subseteq \bd O$, we define the class of test functions that vanish in a neighborhood of $D$,
\begin{align*}
 \C_D^\infty(O) \coloneqq \C_0^\infty(\R^d \setminus D)|_O,
\end{align*}
and we let $\W^{1,2}_D(O)$ be its closure in $\W^{1,2}(O)$. We have $\W_{\bd O}^{1,2}(O) = \W^{1,2}_0(O)$ but on an arbitrary open set the space obtained for $D = \emptyset$ can be a proper subspace of $\W^{1,2}(O)$. 
\section{The form domain \texorpdfstring{$\V$}{V}}
\label{Sec: Form domain}

We shall always assume that the form domain $\V$ in \eqref{eq: sesquilinear form} is one of the following closed subspaces of $\W^{1,2}(O)$:
\begin{itemize}
 \item $\V = \W^{1,2}(O)$ corresponding to Neumann boundary conditions for $L$ or
 \item $\V = \W^{1,2}_D(O)$ for a closed set $D \subseteq \bd O$, corresponding to mixed boundary conditions or, more precisely, Dirichlet conditions on $D$ and Neumann conditions on $\bd O \setminus D$.
\end{itemize}
The second case includes pure Dirichlet boundary conditions ($D = \bd O$) and what is usually called ``good Neumann'' boundary conditions ($D = \emptyset$). In this section, we prove important invariance properties for $V$ and outline their consequences for the semigroup~$T$.

\subsection{Invariance properties}

We shall frequently use the classical result~\cite[Prop.~4.4]{Ouhabaz} that for all $u \in \W^{1,2}(O)$ we have $|u| \in \W^{1,2}(O)$ with 
\begin{align}
\label{eq: derivative modulus}
 \nabla |u| = \Re (\cl{\sgn(u)} \nabla u).
\end{align}
For real-valued $u$, this can be seen as a particular instance of the chain rule, which holds more generally for the composition $\Phi \circ u$ of a real-valued $u \in \W^{1,1}_\loc(O)$ with a Lipschitz function $\Phi: \R \to \R$, see, for example, \cite[Thm.~2.1.11]{Ziemer}. We would like to draw the reader's attention to the particularity that the chain rule does not hold in the same generality for functions $u$ valued in $\IC \cong \R^2$ and Lipschitz functions $\Phi: \IC \to \IC$, see \cite{Leoni-Morini}. What continues to hold for complex valued functions, though, is the bound
\begin{align}
\label{eq: Lipschitz bound composition}
|\nabla(\Phi \circ u)(x)| \leq \mathrm{Lip}(\Phi) |\nabla u(x)|
\end{align}
for almost every $x \in O$. This follows immediately on approximating $u$ by smooth functions in $\W^{1,1}_\loc(O)$, $\Phi$ by smooth functions uniformly on $\IC$, and using the ordinary chain rule for smooth functions.

\begin{lemma}[Invariance properties]
\label{lem: form domain invariance}
\begin{enumerate}
 \item If $\Phi: \IC \to \IC$ is Lipschitz continuous and satisfies $\Phi(0) = 0$, then $\Phi \circ u \in \V$ for all $u \in \V$. 
 \item If $\varphi: \R^d \to \IC$ is Lipschitz continuous and bounded, then $\varphi u \in \V$ for all $u \in \V$.
\end{enumerate}
\end{lemma}

\begin{proof}
As for (i), we first note that \eqref{eq: Lipschitz bound composition} and the pointwise bound $|\Phi(z)| \leq \mathrm{Lip}(\Phi)|z|$ imply $\|\Phi \circ u\|_{1,2} \leq \mathrm{Lip}(\Phi) \|u\|_{1,2}$. If $\V = \W^{1,2}(O)$, then we are done. 

Otherwise, we have $\V = \W^{1,2}_D(O)$ for some closed $D \subseteq \bd O$. By definition, there is a sequence $(u_n)_n \subseteq \C_0^\infty(\R^d \setminus D)$ such that $u_n|_O \to u$ in $\W^{1,2}(O)$ as $n \to \infty$. We set $v_n \coloneqq \Phi \circ u_n$. Then $v_n$ is Lipschitz continuous and thanks to $\Phi(0) = 0$ the properties of having compact support and vanishing in a neighborhood of $D$ carry over from $u_n$ to $v_n$. We conclude $v_n \in \V$ since the required approximants in $\C_0^\infty(\R^d \setminus D)$ can explicitly be constructed by convolution with smooth, compactly supported kernels. The first part of the proof shows that $(v_n)_n$ is bounded in $\V$. Hence, it admits a subsequence with weak limit $v_\infty \in \V$. On the other hand, we have $\|v_n - \Phi \circ u\|_2 \leq \mathrm{Lip}(\Phi) \|u_n - u\|_2$, so that $v_n \to \Phi \circ u$ strongly in $\L^2(O)$ as $n \to \infty$. Thus, we have $\Phi \circ u = v_\infty \in \V$ as required.

As for (ii), we obtain $\|\varphi u\|_{1,2} \leq (\|\varphi\|_\infty + \|\nabla \varphi\|_\infty) \|u\|_{1,2}$ from the product rule and the rest of the proof follows the pattern of (i).
\end{proof}

\subsection{Embedding properties}
\label{Subsec: embedding properties}

These properties will only be relevant in dimension $d \geq 3$, in which case we denote by $2^* \coloneqq 2d/(d-2)$ the Sobolev conjugate of $2$.

\begin{definition}
\label{def: embedding properties of V}
Let $d\geq3$. If $\|v\|_{2^*} \lesssim \|v\|_{1,2}$ holds for all $v \in \V$, then $\V$ has the \emph{embedding property}. It has the \emph{homogeneous embedding property} if $\|v\|_{2^*} \lesssim \|\nabla v\|_2$ holds for all for all $v \in \V$.
\end{definition}

The space $\V = \W^{1,2}_{0}(O)$ always has the homogeneous embedding property~\cite[Thm.~2.4.1]{Ziemer}. A convenient way to approach more general $\V$ is by means of a bounded extension operator $E: \V \to \W^{1,2}(\R^d)$. Such operator exists, for example, under the assumption that $O$ is bounded and $\bd O$ admits Lipschitz coordinate charts around the Neumann boundary part~\cite[Sec.~6]{Hardy}. The embedding property then follows from the commutative diagram
\begin{equation*}
\begin{tikzcd}[row sep=large, column sep=huge]
\W^{1,2}(\R^d)
\arrow{r}{\subseteq} 
& 
\L^{2^*}(\R^d)
\arrow{d}{|_O}
\\
\V 
\arrow{u}{E}
\arrow{r}{\subseteq}
& \L^{2^*}(O)
\end{tikzcd}
\end{equation*}
If $O$ is bounded and connected, then the embedding property in case of mixed boundary conditions already entails the seemingly stronger homogeneous version. This is due to the following lemma, the proof of which relies on an idea that we found in~\cite[Sec.~7]{Daners}

\begin{lemma}
\label{lem: Poincare}
Let $d \geq 3$ and $\V = \W^{1,2}_D(O)$ for a closed set $D \subseteq \bd O$. Assume $O$ is bounded, connected, and that $V$ has the embedding property. Then either $\V$ has the homogeneous embedding property or $\V = \W^{1,2}_\emptyset(O)$  models in fact (good) Neumann boundary conditions.
\end{lemma}

\begin{proof}

\emph{Step 1: $\V \subseteq \L^2(O)$ is compact}.
Suppose $(v_n)_n \subseteq V$ converges to $0$ weakly in $V$. We have to prove that it converges strongly to $0$ in $\L^2(O)$. For any $\varphi \in \C_0^\infty(O)$ with $\|\varphi\|_\infty \leq 1$, we can use Hölder's inequality on $(1-\varphi)v_n$ and the embedding property for $V$ to give
\begin{align*}
 \|v_n\|_2 
 \leq \|\varphi v_n\|_2 + \|(1-\varphi)v_n\|_2
 \lesssim \|v_n\|_2 + |\{x \in O: \varphi \neq 1 \}|^{1/d} \|v_n\|_{1,2}.
\end{align*}
Since $O$ is a bounded set and $(v_n)_n$ is bounded in $V$, the second term on the right can be made as small as we want through the choice of $\varphi$. Once $\varphi$ is fixed, we note that the first term tends to $0$ as $n \to \infty$ since multiplication by $\varphi$ is a bounded operator $V \to \W_0^{1,2}(O)$ and the latter compactly embeds into $\L^2(O)$, see \cite[Thm.~2.5.1]{Ziemer}.

\emph{Step 2: Poincaré alternative}. If we have a Poincaré inequality $\|v\|_2 \lesssim \|\nabla v\|_2$ for all $v \in \V$, then the embedding property trivially implies the homogeneous one. 

Now assume that this was not the case. Then there exists for every $n \in \IN$ some $v_n \in \V$ such that $\|v_n\|_2 = 1$ but $\|\nabla v_n \|_2 \leq 1/n$. Since $(v_n)_n$ is bounded, it admits a subsequence denoted again by $(v_n)_n$ that converges weakly in $\V$ to some $v_\infty \in \V$. Since $\nabla: \V \to \L^2(O)^d$ is bounded, we have $\nabla v_n \to \nabla v_\infty$ weakly in $\L^2(O)^d$ as $n \to \infty$. At the same time, this sequence converges strongly to $0$, meaning that $\nabla v_\infty = 0$. Since $O$ is connected, we conclude that $v_\infty \in \V$ is constant. By compactness of the embedding $\V \subseteq \L^2(O)$ there is another subsequence that converges to $v_\infty$ strongly in $\L^2(O)$ and we conclude $\|v_\infty\|_2 = 1$. 

So far we know that $\V$ contains a nonzero constant function. Since $\V$ is invariant under multiplication with $\C_0^\infty(\R^d)$ functions, see Lemma~\ref{lem: form domain invariance}, we conclude that it contains $\C_{\emptyset}^\infty(O) = \C_0^\infty(\R^d)|_O$. Since $\V$ is closed for the $\W^{1,2}(O)$ norm, we obtain $\W^{1,2}_\emptyset \subseteq \V$. But the reverse inclusion holds by definition of $\V$, and hence we have $\W^{1,2}_\emptyset = \V$.
\end{proof}

\subsection{Off-diagonal estimates}
\label{Subsec: OD}

Lemma~\ref{lem: form domain invariance}.(ii) along with Davies' perturbation method readily yields the following off-diagonal estimates for the semigroup $T$. For convenience, we include the short argument and give explicit constants.

\begin{proposition}
\label{prop: off-diagonal}
Let $\psi \in [0, \pi/2-\omega)$. For all measurable sets $E,F \subseteq O$, all $z \in \S_\psi^+$, and all $f \in \L^2(O)$ with support in $E$ it follows
\begin{align*}
 \|T(z) f\|_{\L^2(F)} \leq \e^{-\frac{\dist(E,F)^2}{4C|z|}}\|f\|_{\L^2(E)},
\end{align*}
where $C = \Lambda + (\Lambda^2 \cos(\omega))/(\lambda \cos(\psi +\omega))$.
\end{proposition}

\begin{proof}
We begin with off-diagonal bounds for $z = t >0$. Let $\varphi: \R^d \to \R$ be bounded and Lipschitz continuous with $\|\nabla \varphi \|_\infty \leq 1$ and let $\rho >0$; both yet to be specified. Then $\e^{\pm \rho \varphi}$ are also bounded and Lipschitz continuous. Since according to Lemma~\ref{lem: form domain invariance} the form domain $\V$ is invariant under multiplication with $\e^{\pm \rho \varphi}$, we can define $\e^{\rho \varphi} L \e^{- \rho \varphi}$ by means of the sesquilinear form
\begin{align*}
 b: \V \times \V \to \IC, \quad b(u,v) = a(\e^{-\rho \varphi}u, \e^{\rho \varphi} v).
\end{align*}
We multiply out the expression for $a(\e^{-\rho \varphi}u, \e^{\rho \varphi} u)$ and then use the ellipticity estimates $|a(u,u)| \leq \Lambda \|\nabla u\|_2^2$, $\Re a(u,u) \geq \lambda \|\nabla u\|_2^2$, and Young's inequality with $\eps \in (0,1)$ to give
\begin{align*}
|b(u,u)| \leq 2 \Lambda \|\nabla u\|_2^2 +  2 \Lambda \rho^2 \|u\|_2^2
\end{align*}
and
\begin{align*}
\Re b(u,u) \geq (1-\eps) \lambda \|\nabla u\|_2^2 - \bigg(\Lambda + \frac{\Lambda^2}{\eps \lambda}\bigg) \rho^2 \|u\|_2^2. 
\end{align*}
We put $C \coloneqq \Lambda + \Lambda^2/(\eps^2 \lambda)$ and conclude that $\e^{\rho \varphi} L \e^{- \rho \varphi} + C \rho^2$ is again a maximal accretive operator on $\L^2(O)$, and hence its negative generates a $C_0$-semigroup of contractions $(S(t))_{t \geq 0}$. In fact, this is the transformed semigroup $S(t) = \e^{-C \rho^2t} \e^{\rho \varphi} T(t) \e^{- \rho \varphi}$, see \cite[II.2.a]{Engel-Nagel}. Now, we specialize to $\varphi(x) = \dist(x,E) \wedge n$ for some $n \in \IN$. For $f \in \L^2(O)$ supported in $E$ we obtain 
\begin{align*}
 T(t)f = \e^{C \rho^2 t} \e^{-\rho \varphi} S(t) \e^{\rho \varphi} f = \e^{C \rho^2 t} \e^{-\rho \varphi} S(t)f
\end{align*}
and consequently
\begin{align*}
 \|T(t)f\|_{\L^2(F)} \leq \e^{C \rho^2 t-\rho (\dist(E,F) \wedge n)} \|f\|_{\L^2(E)}.
\end{align*}
It remains to optimize the parameters $\rho$, $n$, and $C=C(\eps)$. For $\rho = (\dist(E,F) \wedge n)/(2Ct)$, the argument of the exponential function takes its global minimum $-(\dist(E,F) \wedge n)^2/(4Ct)$. Then we pass to the limits as $n \to \infty$ and $\eps \to 1$ and the claim for $z=t > 0$ follows.

In the general case $z \in \S^+_\psi$, we replace $L$ by the operator $\e^{\i \arg z} L$. The associated sesquilinear form $\e^{\i \arg z}a(\cdot, \cdot)$ has the same upper bound $\Lambda$ as $a$ and, by rotation of the numerical range, we obtain the lower bound 
\begin{align*}
 \min_{z' \in \S_\omega^+, \Re z' \geq \lambda} \Re \big(\e^{\i \arg z} z'\big) = \lambda \frac{\cos(\psi + \omega)}{\cos(\omega)} > 0.
\end{align*}
The claim follows from the first part with $t = |z|$ on noting that $\e^{-zL} = \e^{-t \e^{\i \arg z} L}$.
\end{proof}

\section{\texorpdfstring{$\L^p$}{Lp}-dissipativity}
\label{Sec: Dissipativity}

Here, we use the algebraic properties of matrices that satisfy Maz'ya and Cialdea's condition $\Delta_p(A) \geq 0$ to prove that $L$ is formally $\L^p$-dissipative. From this, we shall conclude Theorem~\ref{thm: main2}. 

Some considerations will become less technical for $p\in [2,\infty)$. For the results we are after, this will not cause any burden because we can rely on duality and the following observation that already appears in \cite[Prop.~5.8 \& Cor.~5.17.3]{Carbonaro-Dragicevic}. We reproduce the easy proof in order to become acquainted with the definition of $\Delta_p(A)$ from \eqref{eq: Deltap}.

\begin{lemma}
\label{lem: p-ellipticity and duality}
Let $p \in (1,\infty)$. If $A$ is $p$-elliptic, then both $A$ and $A^*$ are $p$-elliptic and $p'$-elliptic. More precisely, for any $A$ it follows that
\begin{align*}
 \mathrm{(i)} \quad \Delta_p(A) = \Delta_{p'}(A), \qquad \mathrm{(ii)} \quad \Delta_p(A^*) = \Delta_{p'}(A^*) \geq \Delta_p(A) (p/p' \wedge p'/p).
\end{align*}

\end{lemma}

\begin{proof}
Let $\xi \in \IC^d$ and $x \in O$. We have $\i \cJ_{p'} \xi = \cJ_p (\i \xi)$, and hence $(A(x) \xi ,\cJ_{p'}\xi) = (A(x) \i \xi ,\cJ_p(\i \xi))$. This yields (i) as well as the equality in (ii) when applied to $A^*$ instead of $A$. Next, we have $pp' \cJ_p \cJ_{p'} \xi = 4 \xi$ and conclude that
\begin{align*}
 \Re (A^*(x) \xi \mid \cJ_{p'} \xi)
 &= \frac{pp'}{4} \Re (A^*(x) \cJ_p \cJ_{p'} \xi \mid \cJ_{p'} \xi)
 = \frac{pp'}{4} \Re (A(x) \cJ_{p'} \xi \mid  \cJ_p \cJ_{p'} \xi).
\end{align*}
To obtain the inequality in (ii), it suffices to remark that $ |\cJ_{p'} \xi|^2 \geq 4(1/p' \wedge 1/p)^2 |\xi|^2$.
\end{proof}

We continue with a purely algebraic calculation.

\begin{lemma}
\label{lem: p-elliptic complex number}
Almost everywhere on $O$, the following holds for all $Z = X + \i Y$, $X,Y \in \R^d$:
\begin{align*}
 \Re ( AZ \mid Z ) - &\bigg(1 - \frac{2}{p} \bigg)^2 \Re (AX \mid X) - \bigg(1 - \frac{2}{p} \bigg) \Im ((A-A^*)X \mid Y) \\
 &\geq \Delta_p(A) \bigg( \frac{2}{p}|X|^2 + \frac{p}{2} |Y|^2 \bigg).
\end{align*}
\end{lemma}

\begin{proof}
We have 
 \begin{align*}
  \Re (AZ \mid Z) = \Re ( AX \mid X) + \Re ( AY \mid Y) + \Im ((A+A^*)X \mid Y)
 \end{align*}
and in the same manner
\begin{align*}
  \frac{1}{2} \Re (A \tilde Z \mid \cJ_{p} \tilde Z) = \frac{1}{p'}\Re ( A \tilde X \mid \tilde X) + \frac{1}{p} \Re ( A \tilde Y \mid \tilde Y) + \Im \bigg(\bigg(\frac{1}{p}A+\frac{1}{p'}A^*\bigg) \tilde X \biggmid \tilde Y \bigg)
 \end{align*}
for all $\tilde X, \tilde Y \in \R^d$ and $\tilde Z = \tilde X + \i \tilde Y$. We apply the second equality to $\tilde X = (2/\sqrt{p})X$ and $\tilde Y = \sqrt{p} Y$, and then use the first equality to deduce
\begin{align}
\begin{split}
\label{eq1: p-elliptic complex number}
  \frac{1}{2} \Re (A \tilde Z \mid \cJ_{p} \tilde Z) 
  &= \frac{4}{pp'}\Re ( A X \mid X) + \Re ( A Y \mid Y) + \Im \bigg(\bigg(\frac{2}{p}A+\frac{2}{p'}A^*\bigg) X \biggmid Y \bigg) \\
  & = \Re (AZ \mid Z) - \bigg(1-\frac{4}{pp'}\bigg) \Re (AX \mid X) \\
  &\quad + \Im \bigg(\bigg(\frac{2}{p} -1 \bigg)AX + \bigg(\frac{2}{p'} -1 \bigg)A^*X \biggmid Y \bigg).
\end{split}
\end{align}
As we have $1-4/(pp') = (1-2/p)^2$ and $2/p' - 1 = 1-2/p$, the right-hand side of \eqref{eq1: p-elliptic complex number} is precisely what we have to bound from below. By definition, the left-hand side of \eqref{eq1: p-elliptic complex number} is a.e.\ on $O$ greater than or equal to
\begin{align*}
 \frac{1}{2} \Delta_p(A) |\tilde Z|^2 = \Delta_p(A) \bigg( \frac{2}{p}|X|^2 + \frac{p}{2} |Y|^2 \bigg) &\qedhere.
\end{align*}
\end{proof}

The following calculation is similar to \cite[Lem.~2.5]{ELSV} and, in fact, also lies at the heart of the matter in \cite{Mazya-Cialdea}.

\begin{lemma}
\label{lem: chain rule identities}
Let $p \in [2,\infty)$. Given $u \in \V$ and $n \in \IN$, define
\begin{align*}
 v = u(\dualp{u} \wedge n), \qquad w = u(|u|^{p-2} \wedge n^2) \qquad 
\end{align*}
as well as
\begin{align*}
 \chi = \ind_{[|u|^{p-2} \geq n^2]}, \qquad \chi_c = \ind_{[|u|^{p-2} < n^2]}.
\end{align*}

Then $v,w \in \V$ and, abbreviating $Z = \cl{\sgn v} \nabla v$, $X = \Re Z$, and $Y = \Im Z$, it follows that
\begin{align*}
 \cl{\sgn v} \nabla u &= \frac{1}{n} \chi Z + \chi_c |v|^{2/p -1}\bigg(Z- \Big(1-\frac{2}{p}\Big)X \bigg), \\
 \cl{\sgn v} \nabla w &= n \chi Z + \chi_c |v|^{1-2/p}\bigg(Z+ \Big(1-\frac{2}{p}\Big)X \bigg).
\end{align*}
\end{lemma}

\begin{proof}
The functions $z \mapsto z(\dualp{z} \wedge n)$ and $z \mapsto z(|z|^{p-2} \wedge n^2)$ are Lipschitz continuous on $\IC$ and vanish for $z=0$. Lemma~\ref{lem: form domain invariance}.(i) asserts $v,w \in \V$. 

Explicit formul\ae \, for their gradient are derived from the product rule and then the chain rule applied to the second factor that only involves the real-valued function $|u|$. (If $p<3$, some additional smoothing is required since $|\cdot|^{p-2}$ fails to be Lipschitz about $0$.) -- Details of this argument have carefully been written down in \cite[Lemma~5.2]{Sobol-Vogt} and the result is as expected:
\begin{align}
\label{eq1: chain rule identities}
 \nabla v &= n \chi \nabla u + \chi_c \dualp{u} \Big(\nabla u + \bigg(\frac{p}{2}-1 \bigg)(\sgn u) \nabla |u| \bigg) \\
\label{eq2: chain rule identities}
 \nabla w &= n^2 \chi \nabla u + \chi_c |u|^{p-2} \bigg(\nabla u + (p-2)(\sgn u) \nabla |u| \bigg).
\end{align}
Since $|v|$ is obtained from $|u| \in \W^{1,2}(O)$ in the same way as $v$ is obtained from $u$, we have
\begin{align*}
 \nabla |v| = n \chi \nabla |u| + \chi_c \frac{p}{2} \dualp{u} \nabla |u|.
\end{align*}
We multiply this equation by $\chi_c (1-2/p)$ to find
\begin{align*}
 \chi_c \bigg(1-\frac{2}{p} \bigg) \nabla |v| = \chi_c \bigg(\frac{p}{2}-1\bigg) \dualp{u} \nabla |u|.
\end{align*}
Together with the identities $\chi_c |v| = \chi_c |u|^{p/2}$ and $\chi_c \sgn(u) = \chi_c \sgn(v)$, which follow right away from the definition of $v$, the previous identity allows us to solve \eqref{eq1: chain rule identities} for
\begin{align}
\label{eq3: chain rule identities}
 \nabla u = \frac{1}{n} \chi \nabla v + \chi_c |v|^{2/p-1} \bigg(\nabla v - \bigg(1-\frac{2}{p}\bigg) (\sgn v) \nabla |v| \bigg).
\end{align}
We use the same identities to rewrite \eqref{eq2: chain rule identities} as
\begin{align*}
 \nabla w = n^2 \chi \nabla u + \chi_c \bigg(|v|^{2-4/p} \nabla u + \bigg(2-\frac{4}{p}\bigg) |v|^{1-2/p} (\sgn v) \nabla |v| \bigg)
\end{align*}
and plug in the right-hand side of \eqref{eq3: chain rule identities} for $\nabla u$ to give
\begin{align}
\label{eq4: chain rule identities}
 \nabla w = n \chi \nabla v + \chi_c |v|^{1-2/p} \bigg(\nabla v + \bigg(1-\frac{2}{p}\bigg) (\sgn v) \nabla |v| \bigg).
\end{align}
In order to conclude, it suffices to multiply \eqref{eq3: chain rule identities} and \eqref{eq4: chain rule identities} by $\cl{\sgn v}$ each, on recalling $\nabla |v| = \Re (\cl{\sgn v} \nabla v) = X$ and $\sgn v \; \cl{\sgn v} = 1$ on the set where $v \neq 0$.
\end{proof}

We combine the previous two lemmata in order to establish the formal $\L^p$-dissipativity of $L$. It will become important that we do not assume $u \dualp{u} \in \W_\loc^{1,2}(O)$ \emph{a priori} as in~\cite[Thm.~2.4]{Dindos-Pipher}. 

\begin{proposition}
\label{prop: p-dissipativity estimate}
Let $p \in [2,\infty)$. If $A$ is $p$-elliptic and $u \in \dom(L) \cap \L^p(O)$ is such that $Lu \in \L^p(O)$, then it follows that $u\dualp{u} \in \V$ and
\begin{align*}
 \Re ( Lu \mid u|u|^{p-2} ) \geq \frac{2\Delta_p(A)}{p} \|\nabla(u \dualp{u})\|_2^2.
\end{align*}
\end{proposition}

\begin{proof}
 Let $v, w$ be defined as in Lemma~\ref{lem: chain rule identities} for some $n \in \IN$ that we shall not need to specify at this stage of the proof. We have $v,w \in \V$ and in particular 
 \begin{align}
 \label{eq1: p-dissipativity estimate}
  \Re (Lu \mid w) = \Re a(u,w) = \int_O \Re (A(x) \nabla u(x) \mid \nabla w(x) ) \; \d x.
 \end{align}
We shall derive a lower bound for the integrand that holds a.e.\ on $O$. We have $\sgn v = \sgn(u)$, which in turn implies $\sgn v \; \cl{\sgn v} \nabla u = \nabla u$. Relying on the notation of Lemma~\ref{lem: chain rule identities}, we first write
\begin{align*}
 ( A \nabla u \mid \nabla w) 
 = (A \; \cl{\sgn v} \nabla u \mid \cl{\sgn v} \nabla w).
\end{align*}
Then we insert the two formul\ae \, provided by Lemma~\ref{lem: chain rule identities} on the right, taking into account $\chi \chi_c = 0$ and that the scalar term $|v|^{1-2/p}$ commutes with $A$, to give
\begin{align*}
( A \nabla u \mid \nabla w)  
= \chi (AZ \mid Z) + \chi_c \bigg(AZ - \Big(1-\frac{2}{p}\Big)AX \biggmid Z + \Big(1-\frac{2}{p}\Big)X \bigg).
\end{align*}
Taking real parts and simplifying ($\chi + \chi_c = 1$), we arrive at
\begin{align*}
 \Re (A \nabla u \mid \nabla w)  
 = \Re ( AZ \mid Z ) - \chi_c \Big(1 - \frac{2}{p} \Big)^2 \Re (AX \mid X)
- \chi_c \Big(1 - \frac{2}{p} \Big) \Im ((A-A^*)X \mid Y).
\end{align*}
Now, we can use Lemma~\ref{lem: p-elliptic complex number} and then standard ellipticity of $A$ to bound the right-hand side from below in order to deduce
\begin{align}
 \label{eq2: p-dissipativity estimate}
\begin{split}
 \Re (A \nabla u \mid \nabla w)  
 &\geq \chi \Re (AZ \mid Z) + \chi_c \Delta_p(A) \bigg( \frac{2}{p}|X|^2 + \frac{p}{2} |Y|^2 \bigg) \\
 & \geq \chi \lambda |Z|^2 + \chi_c \frac{2\Delta_p(A)}{p} |Z|^2,
\end{split}
\end{align}
where in the final step we have also used that $p\geq 2$. We employ this estimate along with $|Z| = |\nabla v|$ on the right-hand side of \eqref{eq1: p-dissipativity estimate} to obtain
\begin{align}
 \label{eq3: p-dissipativity estimate}
 \Re (Lu \mid w) \geq \int_O \chi \lambda |\nabla v|^2 + \chi_c \frac{2\Delta_p(A)}{p} |\nabla v|^2 \; \d x.
\end{align}

At this stage, we prefer to write $v^{(n)}$, $w^{(n)}$, $\chi^{(n)}$, and $\chi_c^{(n)}$, since they all depend on the level of truncation $n$. In the limit as $n \to \infty$, we obtain from Lebesgue's dominated convergence that $w^{(n)} \to u |u|^{p-2}$ in $\L^{p'}$ norm and $v^{(n)} \to u \dualp{u}$ in $\L^2$ norm. Taking into account $\chi^{(n)} + \chi_c^{(n)} = 1$, we get from \eqref{eq3: p-dissipativity estimate} the rough bound
\begin{align*}
 \Re (Lu \mid w^{(n)}) \geq \bigg(\lambda \wedge \frac{2\Delta_p(A)}{p} \bigg) \|\nabla v^{(n)}\|_2^2,
\end{align*}
which entails a uniform bound for $\nabla v^{(n)}$ in $\L^2(O)$. Altogether, we have shown that $(v^{(n)})_n$ is uniformly bounded in $\V$ and thus admits a subsequence with weak limit $v^{(\infty)} \in \V$. But then we must have $v^{(\infty)}  = u \dualp{u}$ and consequently $u \dualp{u} \in \V$. 

Next, we show pointwise convergence $\nabla v^{(n)} \to \nabla (u \dualp{u})$ a.e.\ on $O$. Every function in $\W^{1,2}(O)$ admits a representative that is absolutely continuous on almost all line segments in $O$ parallel to the coordinate axes and whose classical derivatives are representatives for the weak derivatives in the respective direction. See \cite[Thm.~2.1.4]{Ziemer} for this so-called Beppo Levi Property. On a segment parallel to a coordinate axis $x_i$ with this property for $u$, we can use the classical chain rule for the composition of $u$ with the $\C^1$ function $t \mapsto t \dualp{t}$ to compute $\partial_{x_i}(u \dualp{u}) = |u|^{p/2 -1} (\partial_{x_i} u + (p/2 -1) \sgn(u) \partial_{x_i}|u|)$. Since every such segment also admits the property for $u \dualp{u}$, we conclude
\begin{align*}
 \nabla (u \dualp{u}) = |u|^{p/2 -1} \big(\nabla u + (p/2 -1) \sgn(u) \nabla|u|\big)
\end{align*}
on $O$ in the weak sense. Due to \eqref{eq1: chain rule identities}, the right-hand side above is the pointwise limit of $\nabla v^{(n)}$ as $n \to \infty$ a.e.\ on $O$.

Eventually, we can pass to the limit in \eqref{eq3: p-dissipativity estimate} via Fatou's lemma to conclude
\begin{align*}
 \Re (Lu \mid u|u|^{p-2}) 
 &= \lim_{n \to \infty} \Re (Lu \mid w^{(n)}) \\
 &\geq \int_O \liminf_{n\to \infty} \bigg(\chi^{(n)} \lambda |\nabla v^{(n)}|^2 + \chi_c^{(n)} \frac{2\Delta_p(A)}{p} |\nabla v^{(n)}|^2 \bigg) \; \d x\\
 &= \frac{2\Delta_p(A)}{p} \|\nabla (u \dualp{u})\|_2^2. \qedhere
\end{align*}
\end{proof}

In fact, Proposition~\ref{prop: p-dissipativity estimate} will only be used in the next section, whereas here we need a slight variant that also applies when $\Delta_p(A) = 0$. For pure Dirichlet conditions, such variant was obtained in \cite[Prop.~5.8 \& Prop.~7.6]{Carbonaro-Dragicevic} and their proofs could also be adapted.

\begin{corollary}
\label{cor: p-dissipativity estimate}
Let $p \in [2,\infty)$ and suppose $\Delta_p(A) \geq 0$. If $u \in \V$ is such that $u|u|^{p-2} \in \V$, then $\Re a(u, u|u|^{p-2} ) \geq 0$.
\end{corollary}

\begin{proof}
 We repeat the main steps of the proof of Proposition~\ref{prop: p-dissipativity estimate} and use the same notation. Without any assumption but $u \in \V$, we have again \eqref{eq2: p-dissipativity estimate}. Hence, $\Re (A \nabla u \mid \nabla w^{(n)})  \geq 0$ holds a.e.\ on $O$. Since we \emph{assume} $u|u|^{p-2} \in \W^{1,2}(O)$, we can use again the Beppo Levi Property and \eqref{eq2: chain rule identities} to conclude $\nabla w^{(n)} \to \nabla (u|u|^{p-2})$ a.e.\ on $O$ as $n \to \infty$. This implies $\Re (A \nabla u \mid \nabla (u|u|^{p-2}))  \geq 0$ a.e.\ on $O$, whereupon the claim follows by integration.
\end{proof}

The link between formal $\L^p$-dissipativity as in Corollary~\ref{cor: p-dissipativity estimate} and $\L^p$-contractivity of the semigroup generated by $-L$ is provided by a beautiful result of Nittka~\cite{Nittka}. 

We denote by $\Pb$ the orthogonal projection of $\L^2(O)$ onto the $\L^p$ unit ball $B^p \coloneqq \{u \in \L^2(O) \cap \L^p(O) : \|u\|_p \leq 1\}$. This is a convex and closed subset of $\L^2(O)$. According to~\cite[Theorem~4.1]{Nittka}, the following assertions are equivalent:
\begin{enumerate}
 \item $\V$ is invariant under $\Pb$, and $\Re a(u,u|u|^{p-2}) \geq 0$ for every $u \in \V$ satisfying $u|u|^{p-2} \in \V$.
 \item $\|T(t)f\|_p \leq \|f\|_p$ for all $f \in \L^2(O) \cap \L^p(O)$ and all $t \geq 0$.
\end{enumerate}
Moreover, given $f \in \L^2(O) \setminus B^p$ we have for $u \coloneqq \Pb f$ the implicit formula
\begin{align}
\label{eq: Nittkas implicit formula}
 f = u + t u |u|^{p-2}
\end{align}
for a constant $t>0$ depending on $f$ and $u$, see \cite[Thm.~3.3(a)]{Nittka}.

\begin{proposition}
\label{prop: p-elliptic implies Lp contractive}
If $\Delta_p(A) \geq 0$, then $\|T(t)f\|_p \leq \|f\|_p$ holds for all $f \in \L^2(O) \cap \L^p(O)$ and all $t \geq 0$.
\end{proposition}

\begin{proof}
By duality, it suffices to treat the case $p \geq 2$, see Lemma~\ref{lem: p-ellipticity and duality} and \cite[Thm.~VI.2.5]{Kato} for duality theory for $L$.
Of course we intend to use Nittka's result and verify (i). The second part is precisely the statement of Corollary~\ref{cor: p-dissipativity estimate}. The invariance part is obvious for $f \in \V \cap B^p$ since $\Pb$ is a projection. Hence, we can focus on $f \in \V \setminus B^p$, in which case we have the implicit formula \eqref{eq: Nittkas implicit formula}.

We consider the function $\Upsilon: \R \to \R$ defined by $\Upsilon(s) = s + ts |s|^{p-2}$, which is strictly increasing, vanishes at $0$, and is continuously differentiable with derivative $\Upsilon'(s) = 1+t(p-1)|s|^{p-2} \geq 1$. This implies that $\Psi \coloneqq \Upsilon^{-1}$ is continuously differentiable, strictly increasing and that it satisfies $\|\Psi'\|_\infty \leq 1$, $\Psi'(0) = 1$, and $\Psi(0) = 0$. We take absolute values in \eqref{eq: Nittkas implicit formula} to find $|f| = \Upsilon(|u|)$, that is to say, $|u| = \Psi(|f|)$. Therefore, we can rewrite \eqref{eq: Nittkas implicit formula} as $u = \Phi(f)$, with $\Phi: \IC \to \IC$ given by
\begin{align*}
 \Phi(z) = \frac{z}{1+t\Psi(|z|)^{p-2}}.
\end{align*}
We identify $\IC \cong \R^2$ and write $z = x + \i y$ with $x,y \in \R$. Clearly, $\Phi$ vanishes at $z=0$ and has continuous partial derivatives in every $z \neq 0$. If for $|z|$ close to $0$ and $\infty$ we can bound those uniformly, then $\Phi$ will be Lipschitz continuous and $u \in \V$ will follow from Lemma~\ref{lem: form domain invariance}.

To this end, we begin with a direct calculation for $z \neq 0$ leading to
\begin{align*}
 \frac{1}{2} \Big(|\partial_x \Phi(z)| + |\partial_y \Phi(z)| \Big)
& \leq \frac{1}{1 + t \Psi(|z|)^{p-2}} + \frac{t(p-2)|z| \Psi(|z|)^{p-3} \Psi'(|z|)}{(1 + t \Psi(|z|)^{p-2})^2} \\
&\leq 1 + t(p-2)\frac{|z| \Psi(|z|)^{p-3}}{(1 + t \Psi(|z|)^{p-2})^2},
\end{align*}
where we have used $\Psi(|z|) \geq 0$ and $\Psi'(|z|) \leq 1$ in the second step. For $s \geq 1$ we obtain $\Upsilon(s) \leq (1+t)s^{p-1}$, which entails $|z| \leq (1+t) \Psi(|z|)^{p-1}$ for $|z| \geq \Upsilon(1)$. In this case we have
\begin{align*}
 \frac{1}{2} \Big(|\partial_x \Phi(z)| + |\partial_y \Phi(z)| \Big)
 \leq 1 + \frac{(p-2)|z|}{t\Psi(|z|)^{p-1}}
 \leq 1 + \frac{(p-2)(1+t)}{t}.
\end{align*}
For $|z|$ sufficiently small, we have $|\Psi(z)| \approx |z|$ thanks to $\Psi(0) = 0$ and $\Psi'(0) = 1$ and thus
\begin{align*}
 \frac{1}{2} \Big( |\partial_x \Phi(z)| + |\partial_y \Phi(z)| \Big) \lesssim 1 + |z|^{p-2}. & \qedhere
\end{align*}
\end{proof}

Theorem~\ref{thm: main2} follows from Proposition~\ref{prop: p-elliptic implies Lp contractive} by a routine interpolation argument. 

\begin{proof}[Proof of Theorem~\ref{thm: main2}]
The first part follows from Proposition~\ref{prop: p-elliptic implies Lp contractive}, Lemma~\ref{lem: p-ellipticity and duality}, and the semigroup properties on $\L^2(O)$ by means of Stein's interpolation theorem. The reader can refer to \cite[p.~96]{Ouhabaz} or \cite[Thm.~10.8]{Isem-LectureNotes} for this argument. Extrapolation up to $q=p$ or $q=p'$ under the stronger assumption that $A$ is $p$-elliptic comes for free since the latter is an open-ended condition in $p$, see Section~\ref{Subsec: Consequences}.
\end{proof}
\section{Extrapolation to a holomorphic semigroup}
\label{Sec: Extension of the interval}

We turn toward the proof of Theorem~\ref{thm: main1}. We recall that $\pi/2-\omega$ is the angle of holomorphy of the semigroup $T$ generated by $-L$ on $\L^2(O)$. 

\begin{definition}
\label{def: ultracontractivity}
Let $\psi \in [0, \pi)$. Given $p,q \in (1,\infty)$ with $p \leq q$, a family of operators $(S(z))_{z \in \S_\psi^+} \subseteq \Lop(\L^2(O))$ is said to be \emph{$p \to q$ bounded} if
\begin{align*}
 \|S(z)f \|_{q} \leq C|z|^{d/(2q)-d/(2p)} \|f\|_p 
\end{align*}
holds for some constant $C$ and all $z \in \S_\psi^+$ and all $f \in \L^p(O) \cap \L^2(O)$.
\end{definition}

We start out with a Nash-type inequality.

\begin{lemma}
\label{lem: Nash-type argument}
Assume $d \geq 3$ and that $\V$ has the embedding property. Suppose that $(T(t))_{t>0}$ is $p \to p$ bounded and let $\eps>0$. If $p<2$, then the shifted semigroup $(\e^{-\eps t}T(t))_{t>0}$ is $p \to 2$ bounded and if $p>2$, then it is $2 \to p$ bounded. If $\V$ has the homogeneous embedding property, then the conclusion also holds for $\eps = 0$. 
\end{lemma}

\begin{proof}
By duality it suffices to treat the case $p<2$. For $v \in \V$, we obtain from H\"older's inequality and the embedding property
 \begin{align*}
  \|v\|_2 \leq  \|v\|_{p}^{1-\theta} \|v\|_{2^*}^{\theta} \lesssim  \|v\|_{p}^{1-\theta} \|v\|_{1,2}^{\theta} ,
 \end{align*}
where $\theta \in (0,1)$ is such that $1/2 = (1-\theta)/p + \theta/2^*$. Given $f \in \L^2(O) \cap \L^p(O)$ with $\|f\|_p = 1$, we set $v(t) \coloneqq \e^{-\eps t} T(t) f$ for $t>0$. Ellipticity yields
 \begin{align*}
  \|v(t)\|_{1,2}^2
  \lesssim \Re a(v(t),v(t))+ \eps (v(t)\mid  v(t))
  = \Re ((L+\eps)v(t) \mid v(t))
  = -\frac{1}{2}\frac{\d}{\d t} \|v(t)\|_2^2.
 \end{align*}
By assumption, we have $\|v(t)\|_p \lesssim \e^{- \eps t} \|f\|_p \leq 1$. Combining these three estimates leads us to a differential inequality for $t \mapsto \|v(t)\|_2^2$:
\begin{align*}
 \|v(t)\|_2^{2/\theta} \leq -C (\|v(t)\|_2^2)',
\end{align*}
where $C>0$ depends on the assumptions and ellipticity. If $\V$ has the homogeneous embedding property, then we replace $\|v\|_{1,2}$ by $\|\nabla v\|_2$ in the first line, and hence we can take $\eps = 0$.

If $v$ vanishes somewhere on $(t/2,t)$, then $v(t) = 0$ by the semigroup property and we are done. Otherwise, we obtain 
 \begin{align*}
  \frac{t}{2} \leq - \int_{t/2}^t \frac{C(\|v(s)\|_2^2)'}{\|v(s)\|_2^{2/\theta}} \; \d s \leq \frac{C \theta}{1- \theta} \|v(t)\|_2^{2 - 2/\theta} = \frac{C \theta}{1- \theta} \|v(t)\|_2^{-4p/(2d-pd)},
 \end{align*}
 which is the required $p \to 2$ estimate.
\end{proof}

The following lemma is a blend of ideas found in \cite[Sec.~3]{Pascal-Memoir}. It plays a crucial role for both extrapolating the range of exponents from Theorem~\ref{thm: main2} and obtaining the optimal angle of holomorphy.

\begin{lemma}
\label{lem: ultracontractivity to boundedness}
Let $\eps \geq 0$ and suppose either $p < 2$ and that $(\e^{-\eps t}T(t))_{t>0}$ is $p \to 2$ bounded or suppose $p>2$ and that it is  $2 \to p$ bounded. Then for every $\psi \in [0,\pi/2-\omega)$ and every $q$ between $2$ and $p$ the holomorphic extension $(\e^{-\eps z} T(z))_{z \in \S_\psi^+}$ is $q \to q$ bounded.
\end{lemma}

\begin{proof}
As usual we invoke a duality argument and confine ourselves this time to $p<2$.  We fix some angle $\phi \in (\psi, \pi/2 - \omega)$. Every $z \in \S_\psi^+$ can be written as $z = z' + t$, where $z' \in \S_\phi^+$ and $t>0$ satisfy $|z| \approx |z'| \approx t$ with implied constants depending on $\phi$ and $\psi$. By the semigroup law, we find
 \begin{align*}
  \|\e^{-\eps z} T(z)\|_{p \to 2} \leq \|\e^{-\eps z'} T(z')\|_{2 \to 2} \|\e^{-\eps t} T(t)\|_{p \to 2} \lesssim t^{d/4 - d/(2p)} \approx |z|^{d/4 - d/(2p)}.
 \end{align*}
In particular, for all measurable sets $E, F \subseteq O$ and all $f \in \L^2(O) \cap \L^p(O)$ with support in $E$ we have
\begin{align*} 
 \|\e^{-\eps z} T(z)f\|_{\L^2(F)} \leq  |z|^{d/4 - d/(2p)} \|f\|_{\L^p(E)}.
\end{align*}
Riesz--Thorin interpolation with the off-diagonal bound provided by Proposition~\ref{prop: off-diagonal} yields for all $f \in \L^2(O) \cap \L^q(O)$ with support in $E$ the estimate
\begin{align*}
 \|\e^{-\eps z} T(z)f\|_{\L^2(F)} \lesssim |z|^{d/4 - d/(2q)} \e^{-\frac{\theta \dist(E,F)^2}{4C|z|}}\|f\|_{\L^q(E)}, 
\end{align*}
where $\theta \in (0,1)$ satisfies $1/q = (1-\theta)/p + \theta/2$. This implies $q \to q$ boundedness thanks to the subsequent lemma applied with $g(r) = |z|^{d/4 - d/(2q)} \e^{-\frac{\theta r^2}{4C|z|}}$ and $s=\sqrt{|z|}$.
\end{proof}

We cite \cite[Lem.~4.5]{JDE} with slight change in notation to link it with the above proof.

\begin{lemma}
\label{lem: OD implies bounded}
Let $1 \leq q \leq 2$ and $S$ a bounded linear operator on $\L^2(O)$. Assume that $S$ satisfies off-diagonal estimates in the form
\begin{align*}
 \|S f\|_{\L^2(R \cap O)} \leq g(\dist(Q,R)) \|f\|_{\L^q(Q \cap O)},
\end{align*}
whenever $Q$, $R$ are closed axis-parallel cubes in $\R^d$, $f \in \L^2(O) \cap \L^q(O)$ is supported in $Q \cap O$, and $g$ is some decreasing function. Then $S$ is $q \to q$ bounded with norm bounded by $s^{d/2 - d/q} \sum_{k \in \IZ^d} g(s \max\{|k|/\sqrt{d} -1, 0\})$ for any $s>0$ provided this sum is finite.
\end{lemma}

The next proposition yields Theorem~\ref{thm: main1} up to formalities that we discuss afterward.

\begin{proposition}
\label{prop: p-elliptic generation}
Assume $d \geq 3$ and that $\V$ has the embedding property. Let $p > 2$ and assume that $A$ is $p$-elliptic. Then for every $\psi \in [0, \pi/2 -\omega)$ and every $\eps > 0$ the semigroup $(\e^{-\eps z} T(z))_{z \in \S_\psi^+}$ is $q \to q$ bounded for $q \in (2, dp/(d-2))$. If $\V$ has the homogeneous embedding property, then the same result holds for $\eps = 0$.
\end{proposition}

\begin{proof}
Let $u \in \dom(L) \cap \L^p(O)$ such that $Lu \in \L^p(O)$. According to Proposition~\ref{prop: p-dissipativity estimate}, we have $ v\coloneqq u \dualp{u} \in \V$ along with the bound
\begin{align*}
 \frac{2 \Delta_p(A)}{p} \|\nabla v\|_2^2 \leq \Re (Lu \mid u |u|^{p-2}).
\end{align*}
Hence, we get 
\begin{align}
\label{eq1: p-elliptic generation}
 \frac{2 \Delta_p(A)}{p} \|\nabla v\|_2^2 + \|v\|_2^2 \leq \Re ((1+L)u \mid u |u|^{p-2}).
\end{align}
We use the embedding property of $\V$ on the left and Hölder's inequality on the right to give $\|v\|_{2^*}^2 \lesssim \|(1+L)u\|_p \|u |u|^{p-2}\|_{p'}$. By definition of $v$, this can be rewritten as
\begin{align}
\label{eq2: p-elliptic generation}
 \|u\|_r \lesssim \|(1+L)u\|_p^{1/p} \|u\|_p^{1/p'} \leq (\|u\|_p + \|Lu\|_p)^{1/p}\|u\|_p^{1/p'},
\end{align}
where $r = dp/(d-2)$ and the second step is just the triangle inequality. 

Now, let $f \in \L^2(O) \cap \L^p(O)$ and $t>0$. We intend to apply the previous estimate to $u \coloneqq T(t)f$. Indeed, we have $u \in \dom(L)$ by holomorphy of the semigroup on $\L^2(O)$. Combining Theorem~\ref{thm: main2} and Lemma~\ref{lem: Nash-type argument} leads us to 
\begin{align*}
 \|u\|_p \lesssim \e^{\eps t/2} t^{d/(2p)-d/4}\|f\|_2,
\end{align*}
where $\eps > 0$ is arbitrary but the implied constant depends on $\eps$. The same argument together with holomorphy of the semigroup on $\L^2(O)$ yields
\begin{align*}
 \|Lu\|_p = \|T(t/2)LT(t/2)f\|_p\lesssim \e^{\eps t/2} t^{d/(2p)-d/4}\|LT(t/2)f\|_2 \lesssim \e^{\eps t/2} t^{d/(2p)-d/4-1} \|f\|_2.
\end{align*}
We use these two bounds on the right hand side of \eqref{eq2: p-elliptic generation} to give
\begin{align}
\label{eq3: p-elliptic generation}
 \|T(t)f\|_r \lesssim \e^{\eps t/2} t^{d/(2p)-d/4-1/p} (1+ t)^{1/p}\|f\|_2.
\end{align}
Now, we multiply by $\e^{-\eps t}$ and invoke uniform boundedness of $(1+t)^{1/p}\e^{-\eps t/2}$ for $t > 0$ to ultimately obtain
\begin{align*}
 \|\e^{-\eps t} T(t)f\|_r \lesssim  t^{d/(2p)-d/4-1/p} \|f\|_2.
\end{align*}
As $r = dp/(d-2)$, the exponent of $t$ on the right just happens to be $d/(2r) - d/4$, so that the above estimate proves $2 \to r$ boundedness of $(\e^{-\eps t} T(t))_{t>0}$. We conclude from Lemma~\ref{lem: ultracontractivity to boundedness} that for every $\psi \in [0, \pi/2 - \omega)$ and every $q \in (2,r)$ the holomorphic extension $(\e^{-\eps z} T(z))_{z \in S_\psi^+}$ is $q \to q$ bounded as required.

The modifications if $\V$ even has the homogeneous embedding property are straightforward. In this case we can skip \eqref{eq1: p-elliptic generation} and directly obtain $\|u\|_r \lesssim \|Lu\|_p^{1/p} \|u\|_p^{1/p'}$ as replacement for \eqref{eq2: p-elliptic generation}. From thereon, we can proceed as before, the only exceptions being that we take $\eps = 0$ in Lemmas~\ref{lem: Nash-type argument} and~\ref{lem: ultracontractivity to boundedness} and that the factor $(1+t)^{1/p}$ no longer shows up in~\eqref{eq3: p-elliptic generation}.
\end{proof}

Finally, we give the

\begin{proof}[Proof of Theorem~\ref{thm: main1}]
By duality and Lemma~\ref{lem: p-ellipticity and duality}, it suffices to treat the case $p,q>2$. For such $q$ verifying
\begin{align*}
 \big| 1/2-1/q\big| < 1/d + \big(1-2/d \big) \big| 1/2-1/p\big|,
\end{align*}
the conclusion follows from Proposition~\ref{prop: p-elliptic generation} and holomorphy of the semigroup on $\L^2(O)$ just as in proof Theorem~\ref{thm: main2}. Indeed, we can use Stein interpolation for the restriction of $T$ to any ray $[0,\infty)\e^{\pm \i \psi}$ for $\psi \in [0,\pi/2-\omega)$. The endpoint case for $q$ is again for free since $p$-ellipticity is an open-ended condition.
\end{proof}
\section*{Appendix: An equivalent formulation of \texorpdfstring{$p$}{p}-ellipticity}
We provide a direct proof of a statement from \cite[Prop.~5.15]{Carbonaro-Dragicevic} that we have alluded to in Section~\ref{Subsec: Consequences}.

\begin{tagprop}{A}
\label{prop: p-elliptic via mu}
Let $p \in (1,\infty)$ and suppose that $A: O \to \Lop(\IC^d)$ satisfies \eqref{eq: ellipticity}. Then $\Delta_p(A) > 0$ holds if and only if 
\begin{align*}
 \mu(A) \coloneqq \essinf \limits_{x \in O} \inf \limits_{\xi \in \IC^d, (A(x)\xi \mid \cl{\xi}) \neq 0} \frac{\Re (A(x)\xi \mid \xi)}{|(A(x)\xi \mid \cl{\xi})|} > \big|1-2/p\big|.
\end{align*}
\end{tagprop}

\begin{proof}
Let $\xi = \alpha + \i \beta \in \IC^d$, where $\alpha, \beta \in \R^d$. By definition, we have 
\begin{align*}
 \cJ_p \xi = 2 \bigg(\frac{\alpha}{p'} + \frac{\i \beta}{p} \bigg) = \xi + \bigg(1- \frac{2}{p} \bigg) \cl{\xi}
\end{align*}
and 
\begin{align*}
 \Delta_p(A) = \essinf_{x \in O} \min_{|\xi| = 1} \Re \big( (A(x)\xi \mid \xi) + (1-2/p) (A(x) \xi \mid \cl{\xi}) \big).
\end{align*}
Since the set $\{\xi \in \IC^d : |\xi| = 1\}$ is preserved through multiplication by any $\omega \in \IC$ of modulus $1$, we may replace $\xi$ by $\xi \omega$ and then take the minimum over $\xi$ and $\omega$ to give
\begin{align*}
 \Delta_p(A) 
 &= \essinf_{x \in O} \min_{|\xi| = 1} \min_{|\omega| = 1} \Re \big( (A(x)\xi \mid \xi ) + \omega^2 (1-2/p)  (A(x) \xi \mid \cl{\xi}) \big) .
\end{align*}
Given $x$ and $\xi$, the minimum over $\omega$ is attained precisely when 
\begin{align*}
 \omega^2 (1-2/p)  (A(x) \xi \mid \cl{\xi}) = -|1-2/p| |(A(x) \xi \mid \cl{\xi})|
\end{align*}
and we conclude that 
\begin{align}
\label{eq1: p-elliptic via mu}
 \Delta_p(A) 
 = \essinf_{x \in O} \min_{|\xi| = 1} \big( \Re (A(x)\xi \mid \xi ) - |1-2/p| |(A(x) \xi \mid \cl{\xi})| \big).
\end{align}
Owing to \eqref{eq: ellipticity} we have for a.e.\ $x \in O$ a uniform upper bound $|(A(x) \xi \mid \cl{\xi})| \leq \Lambda$ for all $\xi \in \IC^d$. Hence, if $\Delta_p(A) > 0$ holds, then so does
\begin{align*}
\essinf_{x \in O} \inf_{|\xi| = 1, (A(x)\xi \mid \cl{\xi}) \neq 0} \frac{\Re (A(x)\xi \mid \xi ) - |1-2/p| |(A(x) \xi \mid \cl{\xi})|}{|(A(x) \xi \mid \cl{\xi})|} > 0.
\end{align*}
Since the function we are minimizing is $\R$-homogeneous of degree $0$ in $\xi$, this precisely means $\mu(A)> |1-2/p|$. Conversely, let us assume that $\mu(A) > |1-2/p|$ holds. On passing to the inverse, this is the same as having
\begin{align*}
 \essinf \limits_{x \in O} \inf \limits_{\xi \in \IC^d, (A(x)\xi \mid \cl{\xi}) \neq 0} \frac{\Re (A(x)\xi \mid \xi)- |1-2/p| |(A(x)\xi \mid \cl{\xi})|}{|1-2/p| \Re (A(x)\xi \mid \xi)} > 0.
\end{align*}
Due to \eqref{eq: ellipticity}, we have for a.e.\ $x \in O$  the uniform bound $|1-2/p| \Re (A(x)\xi \mid \xi) \geq |1-2/p| \lambda$ for all $\xi \in \IC^d$. Hence, the latter condition implies that the right-hand side in \eqref{eq1: p-elliptic via mu} is strictly positive and the proof is complete.
\end{proof}
\def\cprime{$'$} \def\cprime{$'$} \def\cprime{$'$}

\end{document}